\documentclass[11pt]{amsart}
\usepackage{graphicx}
\usepackage{verbatim}
\usepackage{textcomp}
\usepackage{amssymb}
\usepackage{cite}
\newtheorem{thm}{Theorem}[section]

\newtheorem{lem}{Lemma}[section]
\newtheorem{prop}{Proposition}[section]

\theoremstyle{definition}

\newtheorem*{ack}{Acknowledgment}
\theoremstyle{remark}
\newtheorem{rem}[thm]{Remark}
\numberwithin{equation}{section}
\def\R{\mathbb R}

\def\N{\mathbb N}
\def\S{\mathbb S}

\def\f{\frac}
\def\td{\tilde}
\def\ra{\rightarrow}
\def\pt{\partial}

\begin{document}
\title[Lower volume growth estimates]{Lower volume growth estimates for Self-shrinkers of mean curvature flow}
\author{Haizhong Li  and  Yong Wei}
\address{Department of Mathematical Sciences, Tsinghua University, Beijing, 100084, China.}
\email{hli@math.tsinghua.edu.cn}
\email{wei-y09@mails.tsinghua.edu.cn}

\thanks{The first author was supported by NSFC No. 10971110 and Tsinghua University-K. U. Leuven Bilateral Scientific Cooperation Fund.}
\maketitle

\begin{abstract}
We obtain a Calabi-Yau type volume growth estimates for complete noncompact self-shrinkers of the mean curvature flow, more precisely, every complete noncompact properly immersed self-shrinker has at least linear volume growth.
\end{abstract}

\section{Introduction}

On a complete noncompact Riemannian manifold $M^n$ with nonnegative Ricci curvature, there are two well known theorems on volume growth estimates of geodesic balls. One is the classic Bishop volume comparison theorem (see \cite{L},\cite{SY}) which says the geodesic balls have at most Euclidean growth, i.e., there exists some positive constant $C$ such that
\begin{eqnarray}
\textrm{Vol}(B_{x_0}(r))&\leq&Cr^n
\end{eqnarray}
holds for $r>0$ sufficiently large. The other is a theorem proved by Calabi \cite{Cal} and Yau \cite{Yau} independently, which says the geodesic balls of such manifolds have at least linear volume growth, that is
\begin{eqnarray}
\textrm{Vol}(B_{x_0}(r))&\geq&C r
\end{eqnarray}
holds for some positive constant $C$.

In this paper, we consider the volume growth estimates on self-shrinkers. Note that there are many similarities between self-shrinkers and gradient shrinking solitons. Self-shrinkers give homothetically self-shrinking solutions to mean curvature flow, and describe possible blow ups at a given singularity of the mean curvature flow. While gradient shrinking Ricci solitons also correspond to the self-similar solutions to Hamilton's Ricci flow, and often arise as Type I singularity models.

Before we state our main theorem, we would like to give a roughly brief review about the already known results on volume growth of gradient shrinking Ricci solitons and self-shrinkers.

For an n-dimensional complete noncompact gradient shrinking Ricci soliton $(M,g,f)$ satisfying
\begin{eqnarray}\label{1-1}
R_{ij}+f_{ij}&=&\f 12g_{ij}
\end{eqnarray}
H. -D. Cao and D. Zhou \cite{CZ} proved that it has at most Euclidean volume growth (see also \cite{ChZ}, \cite{Z}). On the lower volume growth estimate, H.-D. Cao and X.P. Zhu \cite{Cao} proved that any complete noncompact gradient shrinking Ricci soliton must have infinite volume. In fact, they showed that there is some positive constant $C$ such that $\textrm{Vol}(B_{x_0}(r))\geq C\ln\ln r$ for $r$ sufficiently large. If the Ricci curvature is bounded, Carillo-Ni \cite{CN} showed that the volume grows at least linearly.  If the average scalar curvature satisfies
\begin{eqnarray*}
\f 1{\textrm{Vol}(B(r))}\int_{B(r)}Rdv&\leq &\delta
\end{eqnarray*}
for $\delta<n/2$ and $r$ sufficiently large, then Cao-Zhou \cite{CZ} showed that there exists some positive constant $C$ such that $\textrm{Vol}(B_{x_0}(r))\geq Cr^{n-2\delta}$.  In \cite{MW} O. Munteanu and J. Wang proved the sharp result that every complete noncompact gradient shrinking Ricci soliton has at least linear volume growth, which answered the question asked by Cao-Zhou (\cite{CZ}, \cite{Cao}) and Lei Ni that if a Calabi-Yau type lower volume growth estimate holds complete noncompact gradient shrinking Ricci solitons.

\medskip\noindent
{\bf Theorem A} (Munteanu-Wang \cite{MW}) {\it
Let $(M,g,f)$ be a complete noncompact gradient shrinking Ricci soliton, then for any $x_0\in M$ there exists a constant $C>0$ such that
\begin{eqnarray*}
Vol(B_{x_0}(r)) &\geq& Cr,\qquad \textrm{for all } r>0,
\end{eqnarray*}
where $B_{x_0}(r)$ is the geodesic ball of $M$ of radius $r$ centered at $x_0\in M$.}

\vskip 3mm

For a complete noncompact self-shrinker $X:M^n\ra \R^{n+m}$ satisfying
\begin{eqnarray}\label{1-3}
H&=&-\f 12X^N
\end{eqnarray}
Lu Wang\cite{W} proved that every entire graphical self-shrinker has polynomial volume growth. Then Q. Ding and Y. L. Xin \cite{DX} generalized it and showed that if the immersion is proper, then the self-shrinker has at most Euclidean volume growth. After that, Cheng and Zhou \cite{ChZ} improved Ding-Xin's result and gave a sharp volume growth estimate, they showed that $\textrm{Vol}(B_{x_0}(r))\leq Cr^{n-2\beta}$,
with $\beta\leq \inf |H|^2$, where the ball $B_{x_0}(r)$ is defined by
\begin{equation}\label{I1-5}
    B_{x_0}(r)=\{x\in M: \rho_{x_0}(x)<r\},\quad x_0\in M
\end{equation}
with $\rho_{x_0}(x)=|X(x)-X(x_0)|$ is the extrinsic distance function.

In this paper, we consider the lower volume growth estimates for complete noncompact self-shrinkers, an analogue Munteanu-Wang's result will be proved.

\begin{thm}
Let $X:M^n\ra \R^{n+m}$ be a complete noncompact properly immersed self-shrinker, then for any $x_0\in M$ there exists a constant $C>0$
\begin{eqnarray}
\textrm{Vol}(B_{x_0}(r))&\geq&C r, ,\qquad \textrm{for all } r>0
\end{eqnarray}
where the ball $B_{x_0}(r)$ is defined as \eqref{I1-5}.
\end{thm}
\begin{rem}
Note that this is sharp because the volume of the cylinder self-shrinker $X: \S^{n-1}(\sqrt{2(n-1)})\times \R\ra \R^{n+1}$ grows linearly.
\end{rem}
\begin{ack}
 We would like to thank Prof. Huai-Dong Cao for his help and comments, thank Robert Haslhofer for informing us the paper \cite{MW} of Ovidiu Munteanu and Jiaping Wang. Special thanks go to Ovidiu Munteanu for particularly valuable comments and pointing out an error in \eqref{LSI} in the draft of this paper. We also express our thanks to Prof. Detang Zhou for his interests in our work.

\end{ack}

\section{Preliminary}

For a complete immersed self-shrinker $X:M^n\ra \R^{n+m}$ satisfies \eqref{1-3}, we have
\begin{eqnarray}
  |H|^2+\f 14\Delta|X|^2&=&\f n2\label{2-1}\\
  \nabla |X|^2 &=& 2X^T\label{2-2}
\end{eqnarray}

Note that for a gradient shrinking Ricci soliton which satisfies \eqref{1-1}, we take the trace in \eqref{1-1} and get
\begin{eqnarray}\label{2-3}
R+\Delta f&=& \f n2
\end{eqnarray}

The main idea of this paper is comparing the two equations \eqref{2-1} and \eqref{2-3}, in fact, we can correspond $|H|^2$ to $R$, and $\f 14|X|^2$ to $f$, then exploring the similarities between self-shrinker and gradient shrinking Ricci soliton.

Denote $\rho(x)=|X|$, we have
\begin{equation}\label{eq2-4}
    \nabla\rho=\f{X^T}{|X|} \quad \textrm{and}\quad |\nabla\rho|=\f{|X^T|}{|X|}\leq 1,\qquad \textrm{for }\rho\geq 1
\end{equation}
Denote
\begin{eqnarray}
  && B(r)=\left\{x\in M: \rho(x)<r\right\}\\
  && V(r)=\textrm{Vol}(B(r)) =\int_{B(r)}dv,\quad \eta(r)=\int_{B(r)}|H|^2dv
\end{eqnarray}
Then by the co-area formula (cf. \cite{SY}), we have
\begin{eqnarray}
V(r)&=&\int_0^rds\int_{\pt B(s)}\f 1{|\nabla\rho|}d\sigma\\
V'(r)&=&\int_{\pt B(r)}\f 1{|\nabla\rho|}d\sigma=r\int_{\pt B(r)}\f 1{|X^T|}d\sigma\\
\eta(r)&=&\int_0^rds\int_{\pt B(s)}\f{|H|^2}{|\nabla \rho|}d\sigma=\int_0^r sds\int_{\pt B(s)}\f{|H|^2}{|X^T|}d\sigma\\
\eta'(r)&=&r\int_{\pt B(r)}\f{|H|^2}{|X^T|}d\sigma\label{eq2-8}
\end{eqnarray}

Now we state the following Lemma:
\begin{lem}
Let $X:M^n\ra \R^{n+m}$ be a complete noncompact properly immersed self-shrinker, then
\begin{eqnarray}\label{eq3-1}
nV(r)-rV'(r)&=&2\eta(r)-\f 4r\eta'(r)
\end{eqnarray}
\end{lem}
\begin{proof}
Integrate \eqref{2-1} over $B(r)$, we have by using \eqref{1-3}, \eqref{2-2}, \eqref{eq2-4} -\eqref{eq2-8}
\begin{eqnarray*}
nV(r)-2\int_{B(r)}|H|^2&=&\f 12\int_{B(r)}\Delta|X|^2dv\\
&=&\f 12\int_{\pt B(r)}\nabla |X|^2\cdot\nu d\sigma\\
&=&\f 12\int_{\pt B(r)}\nabla |X|^2\cdot \f{\nabla\rho}{|\nabla \rho|}d\sigma\\
&=&\int_{\pt B(r)}|X^T|d\sigma \\
&=&\int_{\pt B(r)}\f{|X|^2-4|H|^2}{|X^T|}d\sigma\\
&=&rV'(r)-4\int_{\pt B(r)}\f{|H|^2}{|X^T|}d\sigma
\end{eqnarray*}
\end{proof}

\begin{rem}
From the fourth equality in the above proof, we can get
\begin{eqnarray}\label{eq2-5}
\f 1{V(r)}\int_{B(r)}|H|^2&\leq&\f n2
\end{eqnarray}
that is, the average of $|H|^2$ is bounded by $n/2$.
\end{rem}

\begin{lem}
Let $X:M^n\ra \R^{n+m}$ be a complete noncompact properly immersed self-shrinker, then
\begin{equation}
\f {V(r_1)}{r_1^n}-\f {V(r_2)}{r_2^n} \leq  2n\f {V(r_1)}{r_1^{n+2}},\quad \textrm{for }r_1>r_2\geq r_0=\sqrt{2(n+2)}
\end{equation}
\end{lem}
\begin{proof}
Lemma 2.1 implies that
\begin{eqnarray*}
(r^{-n}V(r))'&=&r^{-n-1}(rV'(r)-n V(r))\\
&=&4r^{-n-2}\eta'(r)-2r^{-n-1}\eta(r)
\end{eqnarray*}
Integrating the above equation from $r_2$ to $r_1$, we get
\begin{eqnarray*}
r_1^{-n}V(r_1)-r_2^{-n}V(r_2)&=&\int_{r_2}^{r_1}4s^{-n-2}\eta'(s)ds-\int_{r_2}^{r_1}2s^{-n-1}\eta(s)ds\\
&=&4r_1^{-n-2}\eta(r_1)-4r_2^{-n-2}\eta(r_2)\\
&&+2\int_{r_2}^{r_1}(2(n+2)-s^2)s^{-n-3}\eta(s)ds
\end{eqnarray*}

Choose $r_0=\sqrt{2(n+2)}$, and let $r_1>r_2\geq r_0$.  Since $\eta(r)$ is nonnegative and nondecreasing in $r$, we have
\begin{eqnarray*}
\int_{r_2}^{r_1}(2(n+2)-s^2)s^{-n-3}\eta(s)ds&\leq&\eta(r_2)\int_{r_2}^{r_1}(2(n+2)-s^2)s^{-n-3}ds\\
&\leq&\eta(r_2)\left(-2r_1^{-n-2}+2r_2^{-n-2}\right).
\end{eqnarray*}
Thus
\begin{eqnarray*}
r_1^{-n}V(r_1)-r_2^{-n}V(r_2)&\leq&4r_1^{-n-2}(\eta(r_1)-\eta(r_2))\\
&\leq& 4r_1^{-n-2}\eta(r_1)\\
&\leq& 2nr_1^{-n-2}V(r_1)
\end{eqnarray*}
where we used \eqref{eq2-5} in the last inequality.  This completes the proof of Lemma 2.2.
\end{proof}

\begin{rem}
Let $r_2=r_0$ and $r=r_1$ sufficiently large in Lemma 2.2, we can obtain that
 \begin{eqnarray*}
 V(r)&\leq & 2r_0^{-n}V(r_0)r^n
 \end{eqnarray*}
 Since $B_{x_0}(r)\subset B(r+|X_0|)$, we have
 \begin{eqnarray*}
Vol(B_{x_0}(r))&\leq & V(r+|X_0|)\leq C(r+|X_0|)^n\leq 2^nCr^n
 \end{eqnarray*}
for $r\geq |X_0|$.  This recovers Ding-Xin's result \cite{DX}, which states that every complete noncompact properly immersed self-shrinker has at most Euclidean volume growth.
\end{rem}

In the last of this section, we recall the Logarithmic Sobolev inequality for submanifolds in Euclidean space, this was shown by K. Ecker in \cite{Ec}.

\begin{prop}[LSI] Let $X:M^n\ra \R^{n+m}$ be an n-dimensional submanifold with measure $dv$, then the following inequality
\begin{eqnarray}\label{LSI}
&&\int_Mf^2(\ln f^2)e^{-\f {|X|^2}4}dv -\int_Mf^2\ln\left(\int_Mf^2e^{-\f {|X|^2}4} \right)e^{-\f {|X|^2}4}dv \nonumber \\
&\leq&2\int_M|\nabla f|^2e^{-\f {|X|^2}4}dv +\f 12\int_M|H+\f 12X^N|^2f^2e^{-\f {|X|^2}4}dv \\
&&\qquad+C(n)\int_Mf^2e^{-\f {|X|^2}4}\nonumber
\end{eqnarray}
holds for any nonnegative function $f$ for which all integrals are well-defined and finite, where $C(n)$ is a positive constant depending on $n$.
\end{prop}
On self-shrinker which satisfies \eqref{1-3}, the Logarithmic Sobolev inequalities \eqref{LSI} implies the following two inequalities:
\begin{enumerate}
\item[(1)] For any nonnegative function $f$ which satisfies the  normalization
\begin{eqnarray*}\int_Mf^2e^{-\f {|X|^2}4}dv&=&1
\end{eqnarray*}
the following inequality
\begin{eqnarray}\label{2-8}
\int_Mf^2(\ln f)e^{-\f {|X|^2}4}dv&\leq&\int_M|\nabla f|^2e^{-\f {|X|^2}4}dv +\f 12C(n)
\end{eqnarray}
holds.

\item[(2)] By substituting $f=ue^{\f {|X|^2}8}$ into \eqref{LSI}, we have the following inequality
\begin{equation}\label{LSI2}
    \int_Mu^2\ln u^2-\left(\int_Mu^2\right)\left(\ln\int_Mu^2\right)\leq 4\int_M|\nabla u|^2+C(n)\int_Mu^2
\end{equation}
holds for any nonnegative function $u$ for which all the integrals are well-defined and finite.
\end{enumerate}

\section{Proof of Theorem 1.1}

In order to prove Theorem 1.1, we need the following Lemma which holds for any complete properly immersed submanifold in Euclidean space.

\begin{lem}
Let $X:M^n\ra \R^{n+m}$ be a complete properly immersed submanifold.  For any $x_0\in M$, $r\leq 1$, if $|H|\leq \f Cr$ in $ B_{x_0}(r)$ for some positive constant $C>0$, where the ball $B_{x_0}(r)$ is defined as \eqref{I1-5}, then the following inequality holds
\begin{eqnarray}
V_{x_0}(r)=\textrm{Vol}(B_{x_0}(r)) &\geq& \kappa r^n
\end{eqnarray}
here $\kappa=\omega_ne^{-C}$.
\end{lem}

\begin{proof}
In $B_{x_0}(r)$ we have
\begin{eqnarray*}
\Delta \rho_{x_0}^2(x) &=& 2n+2<X-X_0,H>\\
&\geq &2n-2|H|\rho_{x_0}(x)
\end{eqnarray*}
If $|H|\leq \f Cr$ in $B_{x_0}(r)$, then in $B_{x_0}(r)$ we have
\begin{eqnarray}
2n-2\f Cr\rho_{x_0}(x)&\leq&\Delta \rho_{x_0}^2(x)
\end{eqnarray}
Integrating the above equation over $B_{x_0}(s)$ for $s\leq r$
\begin{eqnarray*}
(2n-2\f Crs)V_{x_0}(s)&\leq& \int_{B_{x_0}(s)}\Delta\rho_{x_0}^2(x)\\
&=&\int_{\pt B_{x_0}(s)}\nabla \rho_{x_0}^2(x)\cdot \nu\\
&=&\int_{\pt B_{x_0}(s)}\nabla \rho_{x_0}^2(x)\cdot \f {\nabla \rho_{x_0}(x)}{|\nabla \rho_{x_0}(x)|}\\
&=& \int_{\pt B_{x_0}(s)} 2\f {|(X-X_0)^T|^2}{|(X-X_0)^T|}\\
&\leq & 2s\int_{\pt B_{x_0}(s)} \f {|X-X_0|}{|(X-X_0)^T|}\\
&=& 2s \int_{\pt B_{x_0}(s)} \f 1{|\nabla \rho_{x_0}|}\\
&=& 2sV_{x_0}'(s),
\end{eqnarray*}
where the last equality is due to the co-area formula. This implies
\begin{eqnarray}
\f {V_{x_0}'(s)}{V_{x_0}(s)} &\geq& \f ns-\f Cr
\end{eqnarray}
Integrating from $\epsilon>0$ to $r$, we have
\begin{eqnarray*}
V_{x_0}(r) &\geq& \f {V_{x_0}(\epsilon)}{\epsilon^n}r^ne^{-\f Cr(r-\epsilon)}
\end{eqnarray*}
Let $\epsilon\ra 0$, by $\lim\limits_{\epsilon\ra 0^+}\f {V_{x_0}(\epsilon)}{\epsilon^n}=\omega_n$, we have
\begin{eqnarray}
V_{x_0}(r) &\geq& \kappa r^n,\qquad (\kappa=\omega_ne^{-C})
\end{eqnarray}
\end{proof}

\begin{rem}
As pointed out to us by Ovidiu Munteanu, Lemma 3.1 also follows from Michael-Simon Sobolev inequality, see page 377 in \cite{MS}.
\end{rem}

\vskip 3mm

Next we will prove that every complete noncompact properly immersed self-shrinker has infinite volume, the argument in the following proof is an adoption of Cao-Zhu's \cite{Cao} proof on that complete noncompact shrinking Ricci solitons have infinite volume.
\begin{lem}
Every complete noncompact properly immersed self-shrinker $X: M^n\ra \R^{n+m}$ has infinite volume
\end{lem}
\begin{proof}
We are going to show that if $M$ has finite volume, then we shall obtain a contradiction to the Logarithmic Sobolev inequality \eqref{2-8}.  We denote the annulus region
\begin{eqnarray*}
A(k_1,k_2)=\left\{x\in M: 2^{k_1}\leq \rho(x)\leq 2^{k_2}\right\},&&
V(k_1,k_2)=\textrm{Vol}(A(k_1,k_2)),
\end{eqnarray*}
here $\rho(x)=|X|$. Since $X:M^n\ra \R^{n+m}$ is complete noncompact properly immersed, $X(M)$ cannot be contained in a compact Euclidean ball $\bar{B}(R)$ with radius $R<+\infty$. Then for $k$ large enough, $A(k,k+1)$ contains at least $2^{2k-1}$ disjoint balls
\begin{equation*}
   B_{x_i}(r)=\{x\in M, \rho_{x_i}(x)<r\},\quad x_i\in M, r=2^{-k}
\end{equation*}
where $\rho_{x_i}(x)=|X(x)-X(x_i)|$ is the extrinsic distance function. Noting that on self-shrinker
\begin{eqnarray}
|H|=\f 12|X^N|&\leq& \f 12|X|\leq 2^k =\f 1r, \qquad \textrm{  in } A(k,k+1)
\end{eqnarray}
thus by Lemma 3.1, each ball $B_{x_i}(r)$ has at least volume $\kappa2^{-kn}$, here $\kappa=\omega_ne^{-1}$. So we have
\begin{eqnarray}\label{4-2}
V(k,k+1)&\geq &\kappa 2^{2k-1}2^{-kn}
\end{eqnarray}

Suppose that $\textrm{Vol}(M)<+\infty$, then for every $\epsilon>0$, there exists a large constants $k_0>0$ such that if $k_2>k_1>k_0$, we have
\begin{eqnarray}\label{4-3}
V(k_1,k_2) &\leq & \epsilon
\end{eqnarray}
and we can also choose $k_1, k_2$ in such a way that
\begin{eqnarray}\label{4-4}
V(k_1,k_2) &\leq & 2^{4n}V(k_1+2,k_2-2)
\end{eqnarray}

In deed, we may first choose $K>0$ sufficiently large, and let $k_1\approx K/2$, $k_2\approx 3K/2$, suppose \eqref{4-4} does not hold, i.e.,
\begin{eqnarray*}
V(k_1,k_2) &\geq & 2^{4n}V(k_1+2,k_2-2)
\end{eqnarray*}
If
\begin{eqnarray*}
V(k_1+2,k_2-2) &\leq & 2^{4n}V(k_1+4,k_2-4)
\end{eqnarray*}
then we are done, otherwise we can repeat this process, after $j$ steps we get
\begin{eqnarray*}
V(k_1,k_2) &\geq & 2^{4nj}V(k_1+2j,k_2-2j)
\end{eqnarray*}
When $j\approx K/4$, \eqref{4-2} implies that
\begin{equation*}
    \textrm{Vol}(M)\geq V(k_1,k_2)\geq 2^{nK}V(K,K+1)\geq \kappa 2^{2K-1}
\end{equation*}
But we have already assumed $\textrm{Vol}(M)$ is finite, so after finitely many steps \eqref{4-4} must hold for some $k_2>k_1$. Thus for any $\epsilon>0$ we can choose $k_1$ and $k_2\approx 3k_1$ such that both \eqref{4-3} and \eqref{4-4} are valid.

Now we are going to derive a contradiction to the Logarithmic Sobolev inequality \eqref{2-8}.  We define a smooth cut-off function $\psi(t)$ by
\begin{equation*}
    \psi(t)=\left\{\begin{array}{ll}
                     1, & 2^{k_1+2}\leq t\leq 2^{k_2-2} \\
                     0, & \textrm{outside }[2^{k_1},2^{k_2}]
                   \end{array}\right.\qquad 0\leq \psi(t)\leq 1,\qquad |\psi'(t)|\leq 1
\end{equation*}
Then let
\begin{eqnarray*}
f(x)&=& e^{L+\f{|X|^2}8}\psi(\rho(x))
\end{eqnarray*}
we can choose $L$ such that
\begin{equation}\label{4-5}
    1=\int_Mf^2e^{-\f {|X|^2}4}=e^{2L}\int_{A(k_1,k_2)}\psi^2(\rho(x))
\end{equation}
By the Logarithmic Sobolev inequality \eqref{2-8} we have
\begin{eqnarray*}
\f 12C(n)&\geq& \int_{A(k_1,k_2)}e^{2L}\psi^2(L+\f{|X|^2}8+\ln \psi)\\
&&\quad-\int_{A(k_1,k_2)}e^{2L}\left|\psi'\nabla\rho+\psi \f{X^T}4\right|^2\\
&\geq& \int_{A(k_1,k_2)}e^{2L}\psi^2(L+\f{|X|^2}8+\ln \psi)\\
&&\quad-2\int_{A(k_1,k_2)}e^{2L}|\psi'|^2-\f 18\int_{A(k_1,k_2)}e^{2L}\psi^2 |X|^2\\
&=&L+\int_{A(k_1,k_2)}e^{2L}\psi^2\ln \psi-2\int_{A(k_1,k_2)}e^{2L}|\psi'|^2\\
&\geq& L-(\f 1{2e}+2)e^{2L}V(k_1,k_2),
\end{eqnarray*}
where we have used $|\nabla\rho(x)|\leq 1$ and the elementary inequality $t\ln t\geq -\f 1e$ for $0\leq t\leq 1$.
Then \eqref{4-4} implies,
\begin{eqnarray}
\f 12C(n)&\geq &L-(\f 1{2e}+2)e^{2L}2^{4n}V(k_1+2,k_2-2)\nonumber\\
&\geq&L-(\f 1{2e}+2)2^{4n}e^{2L}\int_{A(k_1,k_2)}\psi^2(\rho(x))\nonumber\\
&=&L-(\f 1{2e}+2)2^{4n}\label{cn3-9}
\end{eqnarray}
where the last equality is due to \eqref{4-5}. On the other hand, by \eqref{4-3} \eqref{4-5} and $0\leq\psi\leq 1$, we have
\begin{eqnarray}
1&\leq&e^{2L}\epsilon.
\end{eqnarray}
So we can make $L$ arbitrary large by letting $\epsilon>0$ sufficiently small, this contradicts with \eqref{cn3-9} because $C(n)$ is just a universal positive constant depending on $n$. Therefore $M$ must have infinite volume.
\end{proof}

\begin{rem}
In the paper \cite{ChZ}, Xu Cheng and Detang Zhou proved that if the self-shrinker is not properly immersed, then it must also have infinite volume.
\end{rem}

\vskip 2mm

Now we are ready to prove Theorem 1.1.
\begin{proof}[Proof of Theorem 1.1]
We use the similar arguments of Munteanu-Wang's in their proof of Theorem A. First we can choose $c>0$ such that $V(r)>0$ for $r\geq c$. To prove Theorem 1.1, it suffices to show there exists a constant $C>0$ depending only on $n$ such that
\begin{equation}\label{pf1}
    V(r)\geq Cr
\end{equation}
hold for all $r\geq c$. Indeed, if \eqref{pf1} holds, then for $\forall x_0\in M$, since for $r$ sufficiently large,
\begin{equation*}
    B_{x_0}(r)\supset B(r-|X_0|),
\end{equation*}
this implies
\begin{equation}
    V_{x_0}(r)\geq V(r-|X_0|)\geq C(r-|X_0|)\geq \f C2 r
\end{equation}
for $r\geq 2|X_0|$.

\vskip 3mm

Now we are going to prove \eqref{pf1} by contradiction. Assume that for any $\epsilon>0$, there exits $r\geq c$ such that
\begin{eqnarray}\label{pf7}
V(r)&\leq& \epsilon r
\end{eqnarray}
Without loss of generality, we can assume $r\in \N$ and consider the following set:
\begin{eqnarray}
D&:=& \left\{k\in \N: V(t)\leq 2\epsilon t \textrm{ for all integers } r\leq t\leq k\right\}
\end{eqnarray}
Obviously $D\neq {\O}$ because $r\in D$, we want to prove that any integer $k\geq r$ is in $D$.

\vskip 3mm

 For $t\geq c$, we define a function $u$ by
 \begin{equation*}
    u(x)=\left\{\begin{array}{cl}
                  1 & \textrm{in } B(t+1)\setminus B(t) \\
                  t+2-\rho(x) & \textrm{in } B(t+2)\setminus B(t+1)\\
                  \rho(x)-(t-1) & \textrm{in } B(t)\setminus B(t-1)\\
                  0& \textrm{otherwise}
                \end{array}\right.
 \end{equation*}
 Substituting $u(x)$ into the Logarithmic Sobolev inequaltiy \eqref{LSI2}, we obtain
 \begin{equation}\label{pflog}
-\left(\int_Mu^2\right)\ln\left(V(t+2)-V(t-1)\right)\leq C_0\left(V(t+2)-V(t-1)\right)
 \end{equation}
with $C_0=C(n)+4+\f 1{e}$, here we have used $|\nabla\rho(x)|\leq 1$ and the elementary inequality $t\ln t\geq -\f 1e$ for $0\leq t\leq 1$.

\medskip

From Lemma 2.2, we have
\begin{equation}
\f {V(t+1)}{(t+1)^n}-\f {V(t)}{t^n}\leq 2n \f {V(t+1)}{(t+1)^{n+2}},\qquad \textrm{for } t\geq \sqrt{2(n+2)}
\end{equation}
then
\begin{eqnarray*}
V(t+1)&\leq& V(t)\f {(t+1)^n}{t^n}\left(1-\f {2n}{(t+1)^2}\right)^{-1}
\end{eqnarray*}
This implies for $t$ sufficiently large,
\begin{eqnarray*}
V(t+1)-V(t) &\leq& V(t)\left( \f {(t+1)^n}{t^n}\left(1+\f {2n}{(t+1)^2}+O(\f 1{(t+1)^4})\right)-1\right)\\
&\leq& V(t)\left((1+\f 1t)^n-1+\f C{t^2}(1+\f 1t)^{n-2}\right)\\
&\leq&V(t)\f Ct
\end{eqnarray*}
So there exists some constant $C_1(n)$ such that for all $t\geq C_1(n)$,
\begin{eqnarray}
    V(t+1)-V(t)&\leq& \td{C}_1\f {V(t)}t,\quad \textrm{and}\label{pf2}\\
     V(t+1)&\leq& 2V(t)\label{pf3}
\end{eqnarray}
where $\td{C}_1$ depending only on $n$. Combining \eqref{pf2} and \eqref{pf3} gives that for all $t\geq C_1(n)+1$,
\begin{eqnarray}
V(t+2)-V(t-1)&\leq&\td{C}_1 \left(\f {V(t+1)}{t+1}+\f {V(t)}t+\f {V(t-1)}{t-1}\right)\nonumber\\
&\leq&\td{C}_1\left(\f 2{t+1}+\f 1t+\f 1{t}(1+\f 1{C_1(n)})\right)V(t)\nonumber\\
&\leq& C_2\f {V(t)}{t}\label{pf8},
\end{eqnarray}
where $C_2$ depending only on $n$. Note that we can assume $r\geq C_1(n)+1$ for the $r$ satisfying \eqref{pf7}.  In fact, if for any give $\epsilon>0$, all the $r$ which satisfies \eqref{pf7} is bounded above by $C_1(n)+1$, then $V(r)\geq \epsilon r$ holds for any $r>C_1(n)+1$, this implies $M$ has at least linear volume growth.

Then for all integers $r\leq t\leq k$, we have $t\in D$, \eqref{pf8} implies
\begin{eqnarray}
V(t+2)-V(t-1)&\leq& 2C_2\epsilon
\end{eqnarray}
If we choose $\epsilon$ such that $2C_2\epsilon<1 $, and noting that
\begin{eqnarray}
\int_Mu^2 &\geq& V(t+1)-V(t)
\end{eqnarray}
then \eqref{pflog} implies
\begin{equation}
\left(V(t+1)-V(t)\right)\ln (2C_2\epsilon)^{-1}\leq C_0\left(V(t+2)-V(t-1)\right)
\end{equation}
Iterating from $t=r$ to $t=k$ and summing up give that
\begin{equation}
\left(V(k+1)-V(r)\right)\ln (2C_2\epsilon)^{-1}\leq 3C_0V(k+2)\leq 6C_0V(k+1)
\end{equation}
where we used \eqref{pf3} in the last inequality. Therefore
\begin{eqnarray}
V(k+1) &\leq & V(r)\f {\ln (2C_2\epsilon)^{-1}}{\ln (2C_2\epsilon)^{-1}-6C_0}\nonumber\\
&\leq&\epsilon r\f {\ln (2C_2\epsilon)^{-1}}{\ln (2C_2\epsilon)^{-1}-6C_0}\label{pf4}
\end{eqnarray}
We can choose $\epsilon$ small enough such that
\begin{eqnarray}\label{pf6}
\f {\ln (2C_2\epsilon)^{-1}}{\ln (2C_2\epsilon)^{-1}-6C_0}&\leq & 2
\end{eqnarray}
So \eqref{pf4} implies
\begin{eqnarray}\label{pf5}
V(k+1) &\leq & 2\epsilon r, \quad\textrm{for any } k\in D
\end{eqnarray}
Noting that $r\leq k+1$, therefore \eqref{pf5} implies $k+1\in D$. Then by induction we conclude that $D$ contains all the integers $k\geq r$. However \eqref{pf5} implies
\begin{eqnarray*}
V(k) &\leq & 2\epsilon r,\qquad\textrm{ for any integer } k\geq r
\end{eqnarray*}
This implies that $M$ has finite volume, which contradicts with Lemma 3.2. So there exists no such $r>c$ such that $V(r)\leq \epsilon r$ with $\epsilon>0$ chosen in \eqref{pf6}. That is $V(r)\geq \epsilon r$ for $r>c$, and this completes the proof of Theorem 1.1.
\end{proof}

\vskip 3mm

By assuming some condition on $|H|^2$, we can further prove the following result,
\begin{prop}
Let $X:M^n\ra \R^{n+m}$ be a complete properly immersed self-shrinker. Suppose the average norm square of the mean curvature satisfies the upper bound
\begin{equation}\label{1-5}
    \f 1{\textrm{Vol}(B(r))}\int_{B(r)}|H|^2\leq \delta
\end{equation}
for some $\delta<\f n2$ and  $r$ sufficiently large . Then for any $x_0\in M$, there exists some positive constant $C$ such that
\begin{eqnarray}
\textrm{Vol}(B_{x_0}(r))&\geq&  Cr^{n-2\delta}
\end{eqnarray}
\end{prop}
\begin{proof}
Combining the assumption \eqref{1-5} with Lemma 2.1 gives that
\begin{equation}
    (n-2\delta)V(r)\leq rV'(r)
\end{equation}
then
\begin{eqnarray*}
    \f {V'(r)}{V(r)}&\geq& \f {n-2\delta}{r}
\end{eqnarray*}
Integrating from $1$ to $r$ gives
\begin{equation*}
    V(r)\geq V(1)r^{n-2\delta}
\end{equation*}
Since $\textrm{Vol}(B_{x_0}(r))\geq V(r-|X_0|)$ for $r>|X_0|$, we have
 \begin{equation}
\textrm{Vol}(B_{x_0}(r))\geq V(1)(r-|X_0|)^{n-2\delta}\geq (\f 12)^{n-2\delta}V(1)r^{n-2\delta}
\end{equation}
for $r>2|X_0|$.
\end{proof}

\end{document}